\documentclass[12pt]{article}

\usepackage{graphicx}
\usepackage{nameref}
\usepackage[shortlabels,inline]{enumitem}
\usepackage{empheq}
\usepackage[shortlabels]{enumitem}
\setlist[enumerate]{nosep}
\usepackage[doc]{optional}
\usepackage{xcolor}
\usepackage[colorlinks=true,
            linkcolor=refkey,
            urlcolor=lblue,
            citecolor=red]{hyperref}
\usepackage{float}
\usepackage{soul}
\usepackage{graphicx}

\usepackage{placeins} 

\def\DJ{{\fontencoding{T1}\selectfont\char208}}

\definecolor{labelkey}{rgb}{0,0.08,0.45}
\definecolor{refkey}{rgb}{0,0.6,0.0}
\definecolor{Brown}{rgb}{0.45,0.0,0.05}
\definecolor{lime}{rgb}{0.00,0.8,0.0}
\definecolor{lblue}{rgb}{0.5,0.5,0.99}

 \usepackage{mathpazo}

\colorlet{hlcyan}{cyan!30}

\colorlet{hlred}{red!40}

\usepackage{stmaryrd}
\usepackage{amssymb}
\makeatletter
\def\namedlabel#1#2{\begingroup
   \def\@currentlabel{#2}%
   \label{#1}\endgroup
}
\makeatother

\newcommand{\bz}{\ensuremath{\mathbf{z}}}

\newcommand{\nnn}{\ensuremath{{n\in{\mathbb N}}}}
\newcommand{\kkk}{\ensuremath{{k\in{\mathbb N}}}}
\newcommand{\thalb}{\ensuremath{\tfrac{1}{2}}}
\newcommand{\menge}[2]{\big\{{#1}~\big |~{#2}\big\}}

\newcommand{\fenv}[1]%
{\ensuremath{\,\overrightarrow{\operatorname{env}}_{#1}}}
\newcommand{\benv}[1]%
{\ensuremath{\,\overleftarrow{\operatorname{env}}_{#1}}}

\newcommand{\scal}[2]{\left\langle{#1},{#2}  \right\rangle}

\newcommand{\RR}{\ensuremath{\mathbb R}}

\newcommand{\ball}{\ensuremath{\mathbb B}}

\newcommand{\NN}{\ensuremath{\mathbb N}}

\newcommand{\argmax}{\ensuremath{\operatorname*{argmax}}}

\newcommand{\prox}{\ensuremath{\operatorname{Prox}}}

\newcommand{\Id}{\ensuremath{\operatorname{Id}}}

{\begin{list}{}{%
\settowidth{\labelwidth}{\textrm{#1~}}%
\setlength{\leftmargin}{\labelwidth+\labelsep}}}
{\end{list}}
\usepackage{amsthm}
\usepackage[capitalize,nameinlink]{cleveref}
\crefname{equation}{}{equations}
\crefname{chapter}{Appendix}{chapters}
\crefname{item}{}{items}
\crefname{enumi}{}{}

\theoremstyle{definition}
\newtheorem{theorem}{Theorem}[section]

\newtheorem{proposition}[theorem]{Proposition}

\newtheorem{example}[theorem]{Example}
\newtheorem{ex}[theorem]{Example}

\newtheorem{remark}[theorem]{Remark}

\providecommand{\RR}{\mathbb{R}}

\providecommand{\Sign}{\operatorname{Sign}}
\providecommand{\sign}{\operatorname{sign}}

\providecommand{\Id}{\operatorname{{ Id}}}

\providecommand{\NN}{\mathbb{N}}

\providecommand{\Id}{\operatorname{Id}}

\providecommand{\RR}{\mathbb{R}}
\providecommand{\NN}{\mathbb{N}}

\definecolor{myblue}{rgb}{.8, .8, 1}

\allowdisplaybreaks 

\begin{document}

\title{\textsc{
Finding best approximation pairs for two intersections of 
closed convex sets
}}

\author{
Heinz H.\ Bauschke\thanks{
Mathematics, University
of British Columbia,
Kelowna, B.C.\ V1V~1V7, Canada. E-mail:
\texttt{heinz.bauschke@ubc.ca}.},~
Shambhavi Singh\thanks{
Mathematics, University
of British Columbia,
Kelowna, B.C.\ V1V~1V7, Canada. E-mail:
\texttt{sambha@student.ubc.ca}.},~ 
and
Xianfu Wang\thanks{
Mathematics, University
of British Columbia,
Kelowna, B.C.\ V1V~1V7, Canada. E-mail:
\texttt{shawn.wang@ubc.ca}.}
}

\date{October 15, 2021} 
\maketitle

\vskip 8mm

\begin{abstract} 
The problem of finding a best approximation pair 
of two sets, which in turn generalizes the well known 
convex feasibility problem, has a long history that dates back to
work by Cheney and Goldstein in 1959.

In 2018, Aharoni, Censor, and Jiang revisited this problem and
proposed an algorithm that can be used when the two sets 
are finite intersections of halfspaces.
Motivated by their work, we present alternative algorithms
that utilize projection and proximity operators. 
{ Our modeling framework is able to accommodate 
even convex sets.} 
Numerical experiments indicate that these methods are 
competitive and sometimes superior 
to the one proposed by Aharoni et al.\ 

\end{abstract}

{


}

\section{Introduction}

Throughout this paper, we assume that 
\begin{empheq}[box=\fbox]{equation*}
    \text{$Y$ is
    a finite-dimensional real Hilbert space with inner product
    $\scal{\cdot}{\cdot}\colon Y\times Y\to\RR$, }
\end{empheq}
and induced norm $\|\cdot\|$. 
Let $m\in\{1,2,\ldots\}$, set 
$I := \{1,\ldots,m\}$ and suppose that 
\begin{empheq}[box=\fbox]{equation*}
    (\forall i\in I)\;\;
    \text{$A_i$ and $B_i$ are nonempty closed convex subsets of $Y$} 
\end{empheq}
such that 
\begin{empheq}[box=\fbox]{equation*}
A := \bigcap_{i\in I} A_i\neq\varnothing
\;\;\text{and}\;\;
B := \bigcap_{i\in I} B_i\neq\varnothing.
\end{empheq}
(We assume here without loss of generality that there are as many 
sets $A_i$ as $B_j$; otherwise, we can either ``copy'' sets or use
the full space $Y$ itself.)
It will occasionally be convenient to work with the convention
$A_{m+1} = A_1$, $A_{m+2} = A_2$, etc.; or, more formally,
$A_n = A_{1+\mathrm{rem}(n-1,m)}$ and 
$B_n = B_{1+\mathrm{rem}(n-1,m)}$. 
We also assume that the projection operators $P_{A_i}$
and $P_{B_i}$ are ``easy'' to compute while the
projections $P_A$ and $P_B$ are ``hard'' and not readily available 
(unless $m=1$).  
The problem we are interested in is to 
find a \emph{best approximation pair}, i.e., to 
\begin{equation}
\label{e:P}
\text{Find $(\bar{a},\bar{b})\in A \times B$
such that~}
\|\bar{a}-\bar{b}\| = \inf_{(a,b)\in A\times B}\|a-b\|. 
\end{equation}
(Note that this problem is actually a generalization of the famous
\emph{convex feasibility problem} 
which asks to find a point in $A\cap B$
provided that this intersection is nonempty which we do not assume here!) 
This problem has a long history, and the first systematic study was
given by Cheney and Goldstein in 1959 \cite{CG}; see also, e.g., \cite{01},
\cite{02}, { and \cite{BCL}}. 
These works, however, assume that 
{ the projection operators $P_A$ and $P_B$ are explicitly
available, which essentially means that} $m=1$. 
Recently, Aharoni, Censor, and Jiang (see \cite{ACJ}) tackled the general case.
Indeed, assuming that the sets $A_i$ and $B_i$ are \emph{halfspaces}, 
they presented  a new algorithm 
--- which we call \textbf{ACJ} for simplicity --- 
for solving \cref{e:P} where they do not require knowledge of the projectors $P_A$ 
and $P_B$ onto the corresponding polyhedra $A$ and $B$. 

\emph{The purpose of this paper is to provide other approaches to solving 
{\rm \cref{e:P}}. We also  provide the required proximity operators 
as well as numerical comparisons. 
The algorithms considered will rely only on the operators $P_{A_i}$ and $P_{B_i}$
and some other operators that are available in closed form.}

The algorithms presented will work for general closed convex sets, not just polyhedra { as long as the projection operators onto the individual 
sets making up the intersections are available.}
Implementable formulae for the underlying algorithmic operators are provided. 
Numerical experiments, similar to one in Aharoni et al.'s paper, are also performed. 
Our results show that other algorithms should be seriously 
considered for solving \cref{e:P}, especially if $m$ is small.

The remainder of the paper is organized as follows. 
In \cref{sec:model}, we consider reformulations of \cref{e:P} 
that are more amenable to the algorithms discussed in \cref{sec:algos}. 
These algorithms rely on computable formulae which we 
present in \cref{sec:proxes}. 
We present a small example on which these algorithms 
are run and convergence is observed using different metrics in \cref{sec:perfeval}. 
In \cref{ss:known} and \cref{ss:unknown}, we consider 
examples where the solutions are known or unknown, 
respectively. 
We also discuss the (positive) effect of pairing up constraints 
(see \cref{sec:duo}). 
We conclude the paper in \cref{sec:final} with a brief 
summary of our findings.

The notation employed in this paper is fairly standard and follows largely \cite{BC2017} to 
which we also refer the reader for general background material. 
{ For the reader's convenience, let us review some notions from convex analysis
that are of fundamental importance to this paper.
The \emph{indicator function} of a set $C$ is written as 
$\iota_C$; we have $\iota_C(x)= 0$ if $x\in C$ and $\iota_C(x)
=+\infty$ otherwise. The corresponding \emph{distance function}
is $d_C(x)=\inf_{c\in C}\|x-c\|$. 
If $C$ is convex, closed, and nonempty, then for every $x$, 
there exists a unique point $P_C(x)\in C$ such that 
$d_C(x)=\|x-P_C(x)\|$. The corresponding operator $P_C$ is called
the \emph{projection operator} or simply \emph{projector} of $C$. 
For instance, 
if $A$ is the \emph{halfspace} $A=\menge{x\in Y}{\scal{a}{x}\leq\alpha}$, 
where $a\neq 0$ and $\alpha\in\RR$, then 
\begin{equation}
\label{e:halfspace}
P_A(x) = x - \frac{\max\{0,\scal{a}{x}-\alpha\}}{\|a\|^2} a. 
\end{equation}
More generally, if $f$ is a function that is convex, lower semicontinuous,
and proper, then for every $x$, the function 
$y\mapsto f(y)+\thalb\|x-y\|^2$ has a unique minimizer which is the 
celebrated \emph{proximal mapping} or \emph{prox operator} of $f$, 
written $\prox_f$. 
Note that if $f=\iota_C$, then we recover the projection operator $P_C$: 
$\prox_{\iota_C} = P_C$. 
A vector $y$ is a \emph{subgradient} of $f$ at $x$ if
for every $h$, we have $f(x)+\scal{y}{h}\leq f(x+h)$;
the set of all subgradients at $x$ is the \emph{subdifferential}
of $f$ at $x$, written $\partial f(x)$. 
}

\section{Modeling \cref{e:P}}

\label{sec:model}

In the product Hilbert space 
\begin{empheq}[box=\fbox]{equation*}
X := Y\times Y,
\end{empheq}
we define
nonempty closed convex sets by 
\begin{empheq}[box=\fbox]{equation*}
(\forall i\in I)\;\;
C_i := A_i\times B_i,
\end{empheq}
along with their intersection 
\begin{empheq}[box=\fbox]{equation*}
(\forall i\in I)\;\;
C := \bigcap_{i\in I} C_i = A \times B. 
\end{empheq}
Note that the projector onto $C_i$ is still easy to compute; indeed, 
$P_{C_i}\colon (x,y)\mapsto (P_{A_i}x,P_{B_i}y)$.
The problem \cref{e:P} is thus equivalent to 
\begin{equation}
\text{minimize $h(x,y)$ subject to $(x,y)\in C$,}
\end{equation}
where 
\begin{equation}
h(x,y) = \alpha\|x-y\|^p,
\end{equation}
$\alpha>0$, and $p\geq 1$. 
Note that $h$ has full domain and it is convex --- 
but $h$ is \emph{not} strictly convex because it has
many minimizers: $\menge{(x,x)}{x\in Y}$.
Also note that  if $p>1$, then $h$ is differentiable.

The problem \cref{e:P} can thus also be alternatively thought of as 
\begin{equation}
\label{e:210121a}
\text{minimize $h(x,y) + \sum_{i\in I}\iota_{C_i}(x,y)$,}
\end{equation}
which features a nonsmooth objective function.
In the next section, we survey various algorithms that could be used to solve
the problem~\cref{e:P} or its reformulations.
We also consider the case when \cref{e:210121a} is approximated by 
\begin{equation}
\label{e:210121aL}
\text{minimize $h(x,y) + \sum_{i\in I}Ld_{C_i}(x,y)$,}
\end{equation}
for some ``large'' constant $L$. 

\section{Algorithms for solving \cref{e:P}}

\label{sec:algos}

{ In this section, we discuss various algorithms.
The algorithms in \cref{ss:ACJ} and \cref{ss:DR} are able to solve the
original problem exactly while those in the remaining subsections can 
be viewed as attempting to solve a perturbed problem.}

\subsection{The Aharoni--Censor--Jiang Algorithm}

\label{ss:ACJ}

This algorithm was recently proposed by 
Aharoni, Censor, and Jiang in \cite{ACJ}.
We denote their algorithm as \textbf{ACJ}. 
\textbf{ACJ}
builds on the earlier \textbf{HLWB} algorithm.
{ (The letters in \textbf{HLWB} signify relevant 
works 
by  Halpern \cite{Halpern}, 
by Lions \cite{Lions}, 
by Wittmann \cite{Wittmann},
and by Bauschke \cite{Bauschke}; 
the name \textbf{HLWB} was coined by 
Censor in \cite{Censor06}.)  
\textbf{ACJ} can be viewed 
as
}
an alternating version of \textbf{HLWB} to find a solution of \cref{e:P}. 
Here is the description of \textbf{ACJ}. 
First, we fix a sequence $(\lambda_k)_\kkk$ of positive real numbers such that 
\begin{equation}
\label{e:210210a}
\lambda_k\to 0,\quad \sum_\kkk \lambda_k=\infty,\quad \sum_\kkk|\lambda_k-\lambda_{k+m}|<\infty
\end{equation}
and also an increasing (not necessarily strictly though) 
sequence of natural numbers $(n_k)_\kkk$ 
such that
\begin{equation}
\label{e:210210b}
n_k \to \infty
\quad\text{and}\quad
\sup_{k_0\in\NN}\sum_{k>k_0}\prod_{n>n_{k_0}}^{n_k}(1-\lambda_n) <\infty.
\end{equation}
For instance, \cref{e:210210a}--\cref{e:210210b} hold
when $(\forall \kkk)$  $\lambda_k=\frac{1}{k+1}$ 
and $n_k= \lfloor1.1^k\rfloor$ (see \cite[page~512]{ACJ}). 
Next, given our sequence of sets $(A_i)_{i\in\NN}$
and $\nnn$, 
we define the operator
\begin{equation}
Q_{A,n}\colon Y\times Y \to Y
\colon (w,w')\mapsto w_{n}, 
\end{equation}
where $w_0=w'$ and $w_{n}$ 
is computed iteratively via 
\begin{equation}
(\forall i\in\{0,1,\ldots,n-1\})
\quad
w_{i+1} = \lambda_{i+1}w + (1-\lambda_{i+1})P_{A_{i+1}}(w_i).
\end{equation}
The operator $Q_{B,n}$, for $(B_i)_{i\in\NN}$ and $\nnn$, 
is defined analogously. 
Finally, we initialize $(x_0,y_0)\in X\times X$, 
and iteratively update via 
\begin{align}
(\forall\kkk)\quad
(x_{k+1},y_{k+1}) &:= \begin{cases}\big(Q_{B,n_k}(y_k,y_k'),y_k\big),& \text{if $k$ is odd;}\\
\big(x_k,Q_{A,n_k}(x_{k},x_k')\big),&\text{if $k$ is even,}\end{cases}
\end{align}
and where 
$(x'_k,y'_k)_\kkk$ in $X\times X$ is a bounded sequence that 
can either be fixed beforehand, e.g., $(x'_k,y'_k)=(y_0,x_0)$, or dynamically updated using, e.g., 
$(x'_0,y'_0)=(y_0,x_0)$ and $(x'_k,y'_k)=(y_{k-1},x_{k-1})$ for $k\in\{1,2,\ldots\}$.

The main result of \cite{ACJ} yields the convergence of 
$(x_{k},y_k)_\kkk$ to a solution of \cref{e:P} provided that 
$C\neq \emptyset$ and each $A_i$ and $B_i$ is
a \emph{halfspace}. 
{ We refer the reader to \cref{e:halfspace} for the formula
for the projection onto a halfspace which is required by \textbf{ACJ}.}

\begin{remark}
Note that \textbf{ACJ} 
takes into account the order of the sets while
the problem \cref{e:P} does not.
It is a nice feature of \textbf{ACJ} that it works throughout in the 
``small'' space $X\times X$. 
{ On the other hand, we are not aware of any extension of 
\textbf{ACJ} to the case when the sets underlying the intersections 
are not halfspaces. This is an interesting topic for further research.}

\end{remark}

\subsection{Douglas--Rachford Algorithm}

\label{ss:DR}

This algorithm, abbreviated as \textbf{DR}, 
{ can be traced back to the paper by Douglas and Rachford 
\cite{DR} although its relevance to optimization was brought to light 
later in the seminal paper by Lions and Mercier \cite{LM}.
\textbf{DR} can deal with problems of the} form \cref{e:210121a}, and
it implicitly operates in the space $X^{m+1}$. 
First, set
$f_0(x,y) := \alpha\|x-y\|^p$ with $\alpha>0$ and $p\geq 1$, 
as well as 
$f_1 := \iota_{C_1},\ldots,f_m:=\iota_{C_m}$ and $I_0 := \{0\}\cup I$. 
Second, fix a parameter $0<\lambda < 2$ (the default being $\lambda=1$).

Now initialize $\bz_0 := (z_{0,0},z_{0,1},\ldots,z_{0,m}) 
=\big((x_{0,0},y_{0,0}),(x_{0,1},y_{0,1}),\ldots,(x_{0,m},y_{0,m})\big)\in 
X^{m+1}$. 
Given $\bz_k = (z_{k,0},z_{k,1},\ldots,z_{k,m})\in X^{m+1}$,
set
\begin{subequations}
\begin{align}
\bar{z}_k &:= \frac{1}{m+1}\sum_{i\in I_0}z_{k,i}\\
(\forall i\in I_0) \;\;
x_{k,i} &:= \prox_{f_i}(2\bar{z}_k-z_{k,i})\\
(\forall i\in I_0) \;\;
z_{k+1,i} &:= z_{k,i} + \lambda(x_{k,i} - \bar{z}_k)
\end{align}
\end{subequations}
to obtain the update $\bz_{k+1} := (z_{k+1,0},z_{k+1,1},\ldots,z_{k+1,m})$. 

If $i\in I$, then the prox operators corresponding to $f_i$ is simply 
the projector $P_{C_i}$.
{
In particular, if each $C_i=A_i\times B_i$ is the Cartesian product
of two halfspaces, then we may utilize \cref{e:halfspace} twice to compute 
$P_{C_i} = \prox_{f_i}$.
}
The prox operator $\prox_{f_0}$ will be computed in closed form for $p\in\{1,2\}$
in \cref{ss:proxes} below. 
{
When $p=1$, which produced better numerical results, then 
$f_0(x,y)=\alpha\|x-y\|$ and 
\begin{align}
\prox_{f_0}(x,y) = 
(x,y) - \frac{1}{\max\big\{2,\|x-y\|/\alpha\big\}}\big(x-y,y-x\big). 
\end{align}
}
It is well known 
{ (see, e.g., \cite[Proposition~28.7]{BC2017})} that the sequence $(\bar{z}_k)_{k\in\NN}$ will converge to a solution of \cref{e:210121a}, i.e., of \cref{e:P}.

\begin{remark}
The \textbf{DR} approach does not care about the order of the sets presented --- 
unlike, \textbf{ACJ}!
A downside is that it operates in the larger space $X^{m+1}$ 
which can become an issue if $m$ is large.
On the positive side, if $C_1\cap \cdots\cap C_m = \varnothing$, 
then $(\bar{z}_k)_\kkk$  will converge to a minimizer of $f_0$ over the set of least-squares
solutions (see \cite[Corollary~6.8]{130} for further information).
{ Finally, it does not require the constraint sets $C_i$ to be
Cartesian product of halfspaces.}
\end{remark}

\subsection{Dual-Based Proximal Method}
\label{ss:dualthingy}
\label{ss:DPG} 
We largely follow Beck's \cite[Section~12.4.2]{Beck2} 
{ (see also  \cite{BT14} and \cite{CDV} for further background
material)} 
but slightly modify the algorithms 
presented there to give two additional methods for solving \cref{e:P}. 
We will work with  the form given in \cref{e:210121a} 
where $h(x,y)$ needs to be $\varepsilon$-strongly convex, for some $\varepsilon>0$, 
which precludes using $\alpha\|x-y\|^p$ directly. 
However, below we will add 
$\varepsilon\frac{1}{2}(\|x\|^2+\|y\|^2)$ to 
this last function to obtain the required $\varepsilon$-strong convexity.
{
We point out that by adding this energy term and solving the 
corresponding new perturbed optimization problem, the solution obtained
does not solve the original problem exactly. 
}

The first method considered is the Dual Proximal Gradient method, 
which --- following \cite{Beck2} --- we abbreviated as \textbf{DPG}. 
Because the algorithm requires strong convexity of the objective
function, we consider 
{
\begin{equation}
\label{e:210722a}
f_0(x,y):=\alpha\tfrac{1}{2}\|x-y\|^2+\varepsilon
\tfrac{1}{2}\big(\|x\|^2+\|y\|^2\big) \;\;\text{with $\alpha>0$ 
and $\varepsilon>0$}. 
\end{equation}
}
We also set $f_1:=\iota_{C_1},\ldots,f_m:=\iota_{C_m}$. 
Second, fix a parameter $L\geq m/\varepsilon$.

Now initialize $\bz_0:=(z_{0,1},\dots,z_{0,m})= \big((x_{0,1},y_{0,1}),\ldots,(x_{0,m},y_{0,m})\big)\in 
X^{m}$, and update it using
\begin{subequations}
\begin{align}
s_k&:=\sum_{i\in I}z_{k,i}\\
x_{k} &:= \argmax_{w\in X} 
\big[ \scal{w}{s_k}-f_0(w) \,\big]\label{e:DPGit}\\
(\forall i\in I) \;\;
z_{k+1,i} &:= z_{k,i} - \frac{1}{L}x_{k}+
\frac{1}{L}P_{f_i}(x_{k}-Lz_{k,i})\label{e:DPGstep}
\end{align}
\end{subequations}
to obtain $\bz_{k+1}:=(z_{k+1,1},\dots,z_{k+1,m})$. 
{ This is the primal representation of DPG, 
see \cite[page~356]{Beck2}, which is most convenient in our setting.}
Once again, the prox operators corresponding to $f_i$ 
for $i\in I$ are just the projectors $P_{C_i}$. 
{
If the sets $C_i$ are Cartesian products of halfspaces, we may use 
\cref{e:halfspace} to compute $P_{C_i}$.}
The closed form for the argmax operator in \cref{e:DPGit} 
{ is given by 
\begin{align}
x_k = \frac{1}{(2\alpha+\varepsilon)\varepsilon}
\big((\alpha+\varepsilon)u_k+\alpha v_k, (\alpha+\varepsilon)v_k+\alpha u_k \big), 
\quad\text{where $s_k=(u_k,v_k)$.}
\end{align}
}
This formula will be proved in \cref{ss:argmax} below.
For sufficiently small $\varepsilon>0$, 
the primal sequence $(x_k)_\kkk$ approximates a solution of 
\cref{e:210121a} and hence of \cref{e:P} 
{ provided that the relative interiors of the sets $C_i$ form
a nonempty intersection (see \cite[page~362]{Beck2}). 
Note that we do not expect that the primal sequence converges to an 
exact solution of \cref{e:P} because the objective function 
$f_0$ in \cref{e:210722a} is not identical to the one required to tackle
\cref{e:P}.}

An accelerated version of \textbf{DPG}, 
known as Fast Dual Proximal Gradient or simply \textbf{FDPG}, 
applies a FISTA-type acceleration 
{ (see \cite[Section~12.3]{Beck2} and \cite{BT14}.)}
Here is how \textbf{FDPG} proceeds: 
Starting with $\bz_0$ as before, 
initialize $\textbf{w}_0:=\bz_0$, $t_0:=1$, 
and update via
\begin{subequations}
\begin{align}
s_{k}'&:=\sum_{i\in I}w_{k,i}\\
u_{k} &:= \argmax_{v\in X}  
\big[ \scal{v}{s_k'}-f_0(v)\,\big]\\ 
(\forall i\in I) \;\;
z_{k+1,i} &:=  w_{k,i} - \frac{1}{L}u_{k}+\frac{1}{L}P_{f_i}(u_{k}-Lw_{k,i})\label{e:FDPGstep}\\
t_{k+1}&:=\frac{1+\sqrt{1+4t_k^2}}{2}\\
(\forall i\in I) \;\;
w_{k+1,i} &:= z_{k,i+1}+\left(\frac{t_{k}-1}{t_{k+1}}\right)(z_{k,i+1}-z_{k,i})
\end{align}
\end{subequations}
to get the primal sequence of interest
\begin{equation}
x_{k+1}:=\argmax_{v\in X}
\big[\scal{v}{s_{k+1}}-f_0(v)\,\big],\quad
\text{where}\quad s_{k+1}:=\sum_{i\in I}z_{k+1,i}.
\end{equation}
Again, for sufficiently small $\varepsilon>0$, 
the sequence $(x_k)_\kkk$ approximates a solution ``close'' 
to that of \cref{e:210121a} 
{--- but not exactly --- }
and (as a consequence of \cref{e:P})  
provided that the relative interior of $C$ is not empty. 

\begin{remark}
Note that although a smaller $\varepsilon$ ensures a solution 
that is closer to that of the original problem \cref{e:P}, 
it also increases the lower bound for $L$, 
which in turn reduces the step size for each iteration 
as seen in \cref{e:DPGstep} and \cref{e:FDPGstep}, 
and so the speed of convergence reduces as well. 
These algorithms are not affected by the order of the sets presented, 
but like \textbf{DR}, they operate in a ``large'' 
space (here $X^m$). This may become a problem when $m$ is large.
\end{remark}

\subsection{Proximal Distance Algorithm}
\label{ss:PDA}

The Proximal Distance Algorithm, or \textbf{PDA} for short, 
was first introduced by Lange and Keys \cite{LangeKeys}. 
{
It is motivated by the framework of MM algorithms, 
where MM stands for majorize/minimize or for minorize/maximize 
depending on the underlying problem. 
This framework was pioneered by Lange; see, e.g.,
his book \cite{Lange} on this topic.
It can be interpreted as a prox-gradient method 
applied to the function $\tfrac{1}{\rho}h+\tfrac{1}{m}\sum_{i\in I}\thalb
d^2_{C_i}$, where the penalty parameter is in theory driven to $+\infty$ 
(see \cite[Section~5.5]{Lange} for a gentle introduction). 
The parameter $\rho$ has to be carefully driven to infinity. 
We will apply \textbf{PDA} to the problem formulation given by \cref{e:210121a}.
}
Set $h(x,y) := \alpha\|x-y\|$, 
which is $\sqrt{2}\alpha$-Lipschitz 
by \cref{p:prox1}, for $\alpha>0$. 
Also, write $z = (x,y)$.
The \textbf{PDA} with starting point $z_0\in X$
generates a sequence $(z_k)_\kkk$ via
\begin{equation}
\label{e:PDA}
z_{k+1} := \prox_{\rho_k^{-1}h}\Big(\sum_{i=1}^{m}\frac{1}{m}P_{C_i}z_k\Big),
\end{equation}
where 
{
\begin{align}
\prox_{\rho_k^{-1}h}(x,y)
= 
(x,y) - \frac{1}{\max\big\{2,\rho_k\|x-y\|/\alpha\big\}}\big(x-y,y-x\big)
\end{align}
by \cref{e:210120b} and where 
}
$(\rho_k)_{k\in\NN}$ is a sequence of positive (and ``sufficiently large'') parameters. 
{
If the sets $C_i$ are Cartesian products of halfspaces, we may use 
\cref{e:halfspace} to compute $P_{C_i}$.}
Under suitable choices of the parameter sequences, the 
{
sequence
}
$(z_k)_\kkk$ approximates
a solution of \cref{e:210121a}.
Lange and Keys recommend $\rho_k = \min\{(1.2)^k\rho_0,\rho_{\text{max}}\}$ 
but other choices may yield better performance (see \cite[Sections~4 and 5]{LangeKeys} and 
\cite{KZL} for details). 
Keys, Zhou, and Lange also point out 
a Nesterov-style accelerated version of PDA (\textbf{accPDA} for short), 
which proceeds as follows: 
\begin{subequations}
\label{e:accPDA}
\begin{align}
w_k &:= 
z_k + \frac{k-1}{k+2}(z_k - z_{k-1}),\\
z_{k+1} &:= \prox_{\rho_k^{-1}h}\Big(\sum_{i=1}^{m}\frac{1}{m}P_{C_i}w_k\Big). 
\end{align}
\end{subequations}
See \cite[Algorithm~1 and Section~3]{KZL} for further information. 
{ Note that because $\rho_k\leq\rho_{\text{max}}<+\infty$,
both \textbf{PDA} and \textbf{accPDA} find a solution of the penalized but not
of the original problem.}

\subsection{Stochastic Subgradient Descent}

\label{ss:SSD}

{
The roots of stochastic gradient descent can be traced back
to two key papers from the early 1950s co-authored by 
Robbins and Monro \cite{RoMo} and 
by Kiefer and Wolfowitz \cite{KiWo};
see also \cite{BCN} for a recent survey.
The method has since been generalized to many different settings.
We follow largely the presentation in \cite{Beck2}. 
}
Set $f_0(x,y) := \alpha \|x-y\|$ and 
$(\forall i\in I)$ 
$f_i := Ld_{C_i}$, where $L>0$. 
Then $f_0$ is $\sqrt{2}\alpha$-Lipschitz 
{
and (by \cref{e:210126b} below)
\begin{align}
f_0'(z) = f_0'(x,y)= \alpha\big(\sign(x-y),-\sign(x-y)\big), 
\end{align}
where ``$\sign$'' is defined in \cref{e:sS}. 
}
The other functions $f_i$ are $L$-Lipschitz.
Moreover, for $i\in I$, we have 
\begin{equation}
\label{e:210126a}
f_i'(z) = L\sign(z-P_{C_i}z)\in \partial f_i(z)
\end{equation}
by, e.g., \cite[Example~16.62]{BC2017}. 
{
If the sets $C_i$ are Cartesian products of halfspaces, we may use 
\cref{e:halfspace} to compute $P_{C_i}$.}
Now Stochastic Subgradient Descent, which we abbreviate as 
\textbf{SSD} (see \cite[Section~8.3]{Beck2} for further
information), applied to \cref{e:210121aL} generates a sequence via 
\begin{equation}
z_{k+1} := z_k - \eta_kf'_{i_k}(z_k),
\end{equation}
where $(\eta_k)_{\kkk}$ is a sequence of positive parameters 
(typically constant or a constant divided by $\sqrt{k+1}$)
and where $f'_{i_k}(z_k)\in\partial f_{i_k}(z_k)$
where $i_k$ is uniformly sampled from $I_0 := \{0\}\cup I$.

{
Under appropriate conditions, the sequence generated by \textbf{SSD} approximates 
a minimizer of the function $\alpha\|x-y\| + L\sum_{i\in I}d_{C_i}$.
Note that for large $L$, the distance functions converge pointwise 
to the corresponding indicator functions, but they are different for fixed $L$.
In this sense, \textbf{SSD} finds a perturbed but not exact solution of the original problem.
}

\section{Useful operators}

\label{sec:proxes}

In this section, we collect formulae for operators 
that are used later in our numerical experiments. 

\subsection{Prox and (sub)differential operators}\label{ss:proxes}

Denote the standard unit ball by $\ball$: 
$\ball := \menge{y\in Y}{\|y\|\leq 1}$. 
It will be convenient to define the generalized signum functions on $Y$ via 
\begin{equation}
\label{e:sS}
\sign(x) := \begin{cases}
x/\|x\|, &\text{if $x\neq 0$;}\\
0, &\text{if $x=0$}
\end{cases}
\;\;\text{and}\;\;
\Sign(x) := \begin{cases}
\{x/\|x\|\}, &\text{if $x\neq 0$;}\\
\ball,  &\text{if $x=0$.}
\end{cases}
\end{equation}
By \cite[Example~16.32]{BC2017}, we have 
\begin{equation}
(\forall x\in Y)\;\;
\sign(x) \in \Sign(x) = \partial\|\cdot\|(x).
\end{equation}

\begin{proposition}
Let $\alpha>0$. 
Then the prox operator of the function
\begin{equation}
h\colon Y\times Y \to\RR\colon
(x,y)\mapsto \alpha\thalb\|x-y\|^2.
\end{equation}
is given by 
\begin{equation}
\prox_h\colon (x,y) \mapsto
\frac{1}{2\alpha+1}
\big((1+\alpha)x+\alpha y,\alpha x+ (1+\alpha )y \big).
\end{equation}
Moreover, 
$\nabla h\colon (x,y)\mapsto 
\big(\alpha(x-y),\alpha(y-x) \big)$ 
is $(1+2\alpha)$-Lipschitz continuous.
\end{proposition}
\begin{proof}
Set $z = (x,y)\in Y\times Y$ and 
$B \colon Y\times Y\to Y\colon (x,y) \mapsto \sqrt{\alpha}(x-y)$. 
Then
\begin{equation}
h(z) = \thalb\|Bz\|^2 = \thalb\scal{Bz}{Bz} = \thalb \scal{B^*Bz}{z}
\end{equation}
and thus 
$\nabla h = B^*B$.
It follows that 
\begin{equation}
\prox_h = (\Id + B^*B)^{-1}.
\end{equation}
Write $B$ in block matrix form, 
$B = \sqrt{\alpha}\begin{bmatrix} \Id,-\Id\end{bmatrix}$. 
Then $B^* = \sqrt{\alpha}\begin{bmatrix} \Id,-\Id\end{bmatrix}^\intercal$,
\begin{equation}
\nabla h = B^*B = \alpha\begin{bmatrix}
\Id & -\Id \\
-\Id & \Id
\end{bmatrix},
\end{equation}
and 
\begin{equation}
\Id + \nabla h = 
\Id+B^*B = \begin{bmatrix}
(1+\alpha)\Id & -\alpha\Id \\
-\alpha\Id & (1+\alpha)\Id
\end{bmatrix}.
\end{equation}
The largest eigenvalue of the very last matrix is $1+2\alpha$
which implies that $1+2\alpha$ is the sharp Lipschitz constant of $\nabla h$.
Finally, 
\begin{equation}
\prox_h = (\Id + \nabla h)^{-1} = 
\big(\Id+B^*B\big)^{-1} = 
\frac{1}{2\alpha+1}
\begin{bmatrix}
(1+\alpha)\Id & \alpha\Id \\
\alpha\Id & (1+\alpha)\Id
\end{bmatrix}
\end{equation}
and the result follows.
\end{proof}

\begin{proposition}
\label{p:prox1}
Let $\alpha>0$. 
The function 
\begin{equation}
h\colon Y\times Y \to\RR\colon
(x,y)\mapsto \alpha\|x-y\|,
\end{equation}
is $\alpha\sqrt{2}$-Lipschitz 
and a convenient selection of $\partial h$ is given by 
\begin{subequations}
\label{e:210126b}
\begin{align}
(x,y)&\mapsto \alpha\big(\sign(x-y),-\sign(x-y)\big)\\
&= \begin{cases}
\big(\alpha(x-y)/\|x-y\|,\alpha(y-x)/\|y-x\|\big), &\text{if $x\neq y$};\\
\big(0,0\big), &\text{if $x=y$}
\end{cases}\\
&\in \menge{\alpha(s,-s)}{s\in\Sign(x-y)}\\ 
&= \partial h(x,y). 
\end{align}
\end{subequations}
The prox operator of $h$ is given by 
\begin{equation}
\label{e:210120b}
\prox_h\colon (x,y) \mapsto 
(x,y) - \frac{1}{\max\big\{2,\|x-y\|/\alpha\big\}}\big(x-y,y-x\big).
\end{equation}
\end{proposition}
\begin{proof}
Set 
$A = \begin{bmatrix} \Id,-\Id\end{bmatrix}$, 
$z = (x,y)$, 
and $f(z) = \|Az\|$. 
Then $Az=x-y$ and $h = \alpha f$.
Furthermore,
$\|Az\|^2 = \|x-y\|^2
= \|x\|^2 + \|y\|^2 - 2\scal{x}{y}
\leq 2\|x\|^2 + 2\|y\|^2 = 2(\|x\|^2+\|y\|^2) = 2\|z\|^2$
which shows that $f$ is $\sqrt{2}$-Lipschitz and therefore 
$h$ is $\alpha\sqrt{2}$-Lipschitz.
The subdifferential formula follows using 
\cite[Proposition~16.6(i) and Example~16.32]{BC2017}. 
Note that, for $\beta\geq 0$, 
\begin{equation}
\label{e:210119a}
A^\intercal = \begin{bmatrix} \Id\\-\Id\end{bmatrix},
\;\;
AA^\intercal = 2\Id,
\;\;
(AA^\intercal +\beta\Id)^{-1} = \frac{1}{2+\beta}\Id,
\end{equation}
and 
\begin{subequations}
\label{e:210119b}
\begin{align}
\Id - A^\intercal(AA^\intercal +\beta\Id)^{-1}A
&= 
\begin{bmatrix}
\Id & 0 \\
0 & \Id
\end{bmatrix}
- \begin{bmatrix} \Id\\-\Id\end{bmatrix}\frac{1}{2+\beta}\begin{bmatrix} \Id,-\Id\end{bmatrix}\\
&= \begin{bmatrix}
\Id & 0 \\
0 & \Id
\end{bmatrix} 
-\frac{1}{2+\beta}
\begin{bmatrix}
\Id & -\Id \\
-\Id & \Id
\end{bmatrix}\\
&= 
\frac{1}{2+\beta}\begin{bmatrix}
1+\beta& 1 \\
1 & 1+\beta
\end{bmatrix}. 
\end{align}
\end{subequations}

We now discuss cases.

\emph{Case~1:} $\|x-y\|\leq 2\alpha$.\\
In view of \cref{e:210119a} (with $\beta=0$),
this is equivalent to 
$\|(AA^\intercal)^{-1}Az\|\leq\alpha$.
Using \cite[Lemma~6.68]{Beck2} and \cref{e:210119b} (with $\beta=0$), we obtain
(switching back to row vector notation for convenience)
\begin{equation}
\label{e:210119c1}
\prox_h(x,y)
= \frac{1}{2}(x+y,x+y)
= (x,y) - \frac{1}{2}(x-y,y-x). 
\end{equation}

\emph{Case~2:} $\|x-y\|> 2\alpha$.\\
In view of \cref{e:210119a} (with $\beta=0$),
this is equivalent to 
$\|(AA^\intercal)^{-1}Az\|>\alpha$.
Using again \cref{e:210119a} (with general $\beta\geq 0$), 
we set and obtain 
\begin{subequations}
\begin{align}
g(\beta) 
&= \|(AA^\intercal+\beta\Id)^{-1}Az\|^2 - \alpha^2\\
&= \frac{1}{(2+\beta)^2}\|x-y\|^2-\alpha^2.
\end{align}
\end{subequations}
Because $\alpha>0$ and $\beta\geq 0$,
we obtain the equivalences:
$g(\beta)=0$ 
$\Leftrightarrow$
$\|x-y\|=\alpha(2+\beta)$
$\Leftrightarrow$
$\beta= 
\|x-y\|/\alpha- 2$.
Set 
\begin{equation}
\beta^* := \frac{\|x-y\|}{\alpha} - 2 > 0
\end{equation}
so that $g(\beta^*)=0$. 
Moreover, 
$2+\beta^* = \|x-y\|/{\alpha}$ and 
$1+\beta^* = (\|x-y\|-\alpha)/\alpha$. 
Hence \cref{e:210119b} (with $\beta=\beta^*$) turns into 
\begin{subequations}
\label{e:210120a}
\begin{align}
\Id - A^\intercal(AA^\intercal +\beta^*\Id)^{-1}A
&= 
\frac{1}{2+\beta^*}\begin{bmatrix}
1+\beta^* & 1 \\
1 & 1+\beta^*
\end{bmatrix} \\
&=
\frac{\alpha}{\|x-y\|}\begin{bmatrix}
\frac{\|x-y\|-\alpha}{\alpha} & 1 \\
1 & \frac{\|x-y\|-\alpha}{\alpha}
\end{bmatrix}\\
&= 
\begin{bmatrix}
1 & 0 \\ 0 & 1
\end{bmatrix}
- \frac{\alpha}{\|x-y\|}
\begin{bmatrix}
1 & - 1\\
-1 & 1
\end{bmatrix}. 
\end{align}
\end{subequations}
Using \cite[Lemma~6.68]{Beck2} and \cref{e:210120a}, we obtain
(switching back to row vector notation for convenience)
\begin{equation}
\label{e:210119c2}
\prox_h(x,y)
= (x,y)-\frac{\alpha}{\|x-y\|}(x-y,y-x).
\end{equation}
Finally, the formula given in \cref{e:210120b}
follows by combining \cref{e:210119c1} and \cref{e:210119c2}.
\end{proof}

\subsection{Argmax operator}

\label{ss:argmax}

Consider, for $\alpha>0$ and $\varepsilon >0$, 
\begin{equation}
f_0(x,y) := \alpha\thalb\|x-y\|^2 + \varepsilon\thalb\big(\|x\|^2+\|y\|^2),
\end{equation}
which is a perturbation of $\alpha\thalb\|x-y\|^2$
that is $\varepsilon$-strongly convex. 

Given $(u,v)\in X$, 
the dual-based proximal methods of 
\cref{ss:dualthingy} require from us to 
find the unique maximizer of 
\begin{subequations}
\label{e:210211a}
\begin{align}
(x,y)&\mapsto \scal{(u,v)}{(x,y)}  - f_0(x,y)\\
&= \scal{u}{x} + \scal{v}{y} - \alpha\thalb\|x-y\|^2 - \varepsilon\thalb\big(\|x\|^2+\|y\|^2).
\end{align}
\end{subequations}
We now derive an explicit formula for this maximizer: 

\begin{proposition}
\label{p:210211b}
Given $\alpha>0$, $\varepsilon>0$, 
and $(u,v)\in X$, 
the unique maximizer of \cref{e:210211a} is
\begin{equation}
(x,y)
= 
\frac{1}{(2\alpha+\varepsilon)\varepsilon}
\big((\alpha+\varepsilon)u+\alpha v, (\alpha+\varepsilon)v+\alpha u \big).
\end{equation}
\end{proposition}

\begin{proof}
Because $f_0$ is strongly convex, 
we employ standard convex calculus 
to find the maximizer by 
finding the zero of the gradient of the function in 
\cref{e:210211a}. That is, we need to solve 
\begin{align}
(0,0)
&= 
\big(u-\alpha x+\alpha y-\varepsilon x,v-\alpha y+\alpha x-\varepsilon y \big);
\end{align}
or equivalently (switching to more formal column vector notation),
\begin{equation}
\begin{bmatrix} 
  \alpha+\varepsilon & -\alpha\\
  -\alpha & \alpha+\varepsilon 
\end{bmatrix}
\begin{bmatrix}
x\\ y
\end{bmatrix}
= \begin{bmatrix}
u\\v
\end{bmatrix}. 
\end{equation}
Because 
\begin{equation}
  \begin{bmatrix} 
    \alpha+\varepsilon & -\alpha\\
    -\alpha & \alpha+\varepsilon 
  \end{bmatrix}^{-1}
  = 
  \frac{1}{(2\alpha+\varepsilon)\varepsilon}
  \begin{bmatrix} 
    \alpha+\varepsilon & \alpha\\
    \alpha & \alpha+\varepsilon 
  \end{bmatrix}
\end{equation}
we obtain the announced formula.
\end{proof}

\section{Numerical experiments}

\label{sec:perfeval}

We start by discussing metrics --- in the sense of ``standard of measurement'' 
not in the sense of topology --- 
to evaluate the quality of the the iterates 
for the different algorithms proposed in \cref{sec:algos}. 
To make the measurement uniform over the different algorithms, 
we use this metric once for each prox or proj evaluation. 
For the same purpose, 
we consider $A_i$ or $B_i$ as unit inputs in these operators. 
So, for example, as \textbf{DR} projects to all $A_i$ and $B_i$ along with a prox evaluation, we get a total of $2m+1$ operations. 
Since it is not possible to obtain the final output of a given iteration before computing all the proxes, 
we repeat the final output given by the iteration $2m+1$ times.
Again for the sake of uniformity, we only repeat the final output of each iteration, 
regardless of whether intermediate updates can be calculated.
We shall consider two cases: in the first, true solutions are known 
(and assumed to be unique) 
while in the second, they aren't.
The former case allows us to inspect the progress of the iterates towards the solution,
while in the latter case a metric is needed to gauge the performance of the algorithms.
The latter scenario is the one most realistic for applications.

\subsection{Two examples}

\label{ss:known}

\begin{figure}[ht]
\centering
\includegraphics[width=\textwidth]{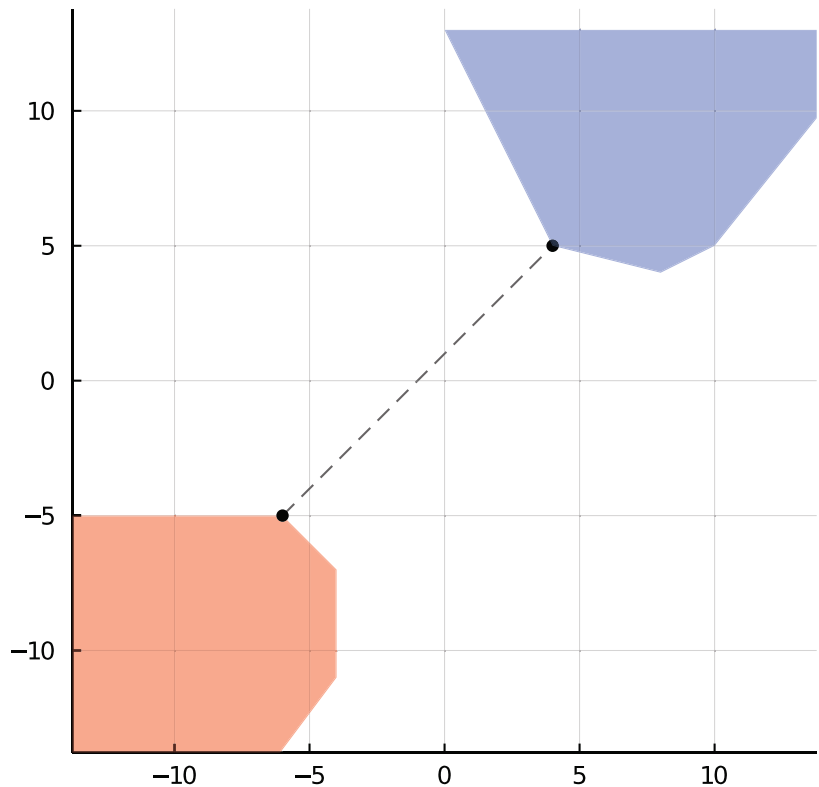}
\caption{Visualization of the sets {A} (red) 
and {B} (blue) for \cref{eg:toy} and the solution pair (black dots) 
for \cref{e:P}.}
\label{f:toyset}
\end{figure}

The convergence plots for the algorithms are straightforward when the solution is known.
Assuming the solution is unique and denoted by  $(\bar{x},\bar{y})$, 
we use the metric 
\begin{equation}
(x,y)\mapsto \|(x,y)-(\bar{x},\bar{y})\|\label{e:knownsol}
\end{equation}
applied to the appropriate iterates of the algorithms.

The distance to the solution \cref{e:knownsol} is evaluated once 
for each projection or prox operator evaluation. 
For example, because each iteration of \textbf{DR} 
(see \cref{ss:DR}) 
uses 
$2m+1$ prox evaluation 
(see the first paragraph of  \cref{sec:perfeval}), 
one \textbf{DR} step invokes 
$2m+1$ ``updates''. 
This approach ensures that the evaluation has some uniformity/fairness 
over the different algorithms. 

We begin with an example where we know the solution.
This (small-scale) example is motivated by the one provided
by Aharoni et al.\ in \cite[Section~5]{ACJ}.

\begin{example}
\label{eg:toy}
Consider the subsets $A$ and $B$ of $Y=\RR^2$,
defined by the two systems of $m=4$ linear inequalities

\begin{figure}[ht]
\centering 
\includegraphics[width=0.9\textwidth]{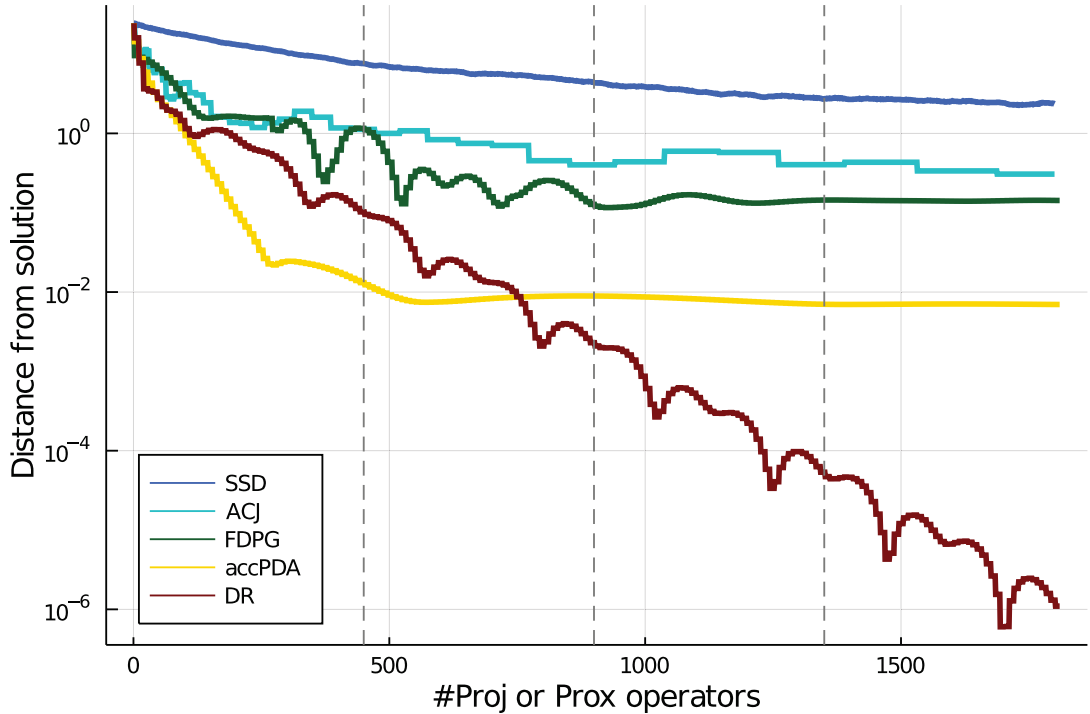}
\caption{Convergence plot for \cref{eg:toy} using the 
known-solution metric
}
\label{f:toycvg}
\end{figure}
\begin{equation}
\left[\begin{array}{rrr}
4&3&-17\\1&0&4\\1&1&11\\0&1&5
\end{array}\right]
\begin{bmatrix}
x_1\\x_2\\ 1
\end{bmatrix}\leq 
\left[\begin{array}{r} 0 \\ 0 \\ 0 \\ 0\end{array}\right]
\quad\text{and}\quad
\left[\begin{array}{rrr}
5&-4&-30\\1&-2&0\\-1&-4&24\\-2&-1&13
\end{array}\right]
\begin{bmatrix}
x_1\\x_2\\ 1
\end{bmatrix}\leq 
\left[\begin{array}{r} 0 \\ 0 \\ 0 \\ 0\end{array}\right],
\end{equation}
respectively. 
The corresponding problem \cref{e:P} possesses
the unique solution 
\begin{equation}
\big( (-6,-5),(4,5) \big),
\end{equation} 
which is also visualized in \cref{f:toyset}. 

We ran the algorithms from \cref{sec:algos}, 
and also accelerated versions when available. 
The algorithms were run for a total 1800 Prox or Projection 
operations per algorithm. 
The accelerated versions performed clearly better than the 
original versions in this case. 
Therefore, for the clarity of the exposition, 
we do not report the \textbf{DPG} and \textbf{PDA} results. 
We used the following parameters and
also report to how many iterations in each algorithm 
this corresponded: 

\begin{enumerate}[leftmargin=0cm,itemindent=2.5cm]
\item[\textbf{ACJ}:] 
55 iterations, with
$\lambda_k={1}/{(k+1)}$, $n_k=\lfloor 1.1^k \rfloor$, and

\hspace{2.5cm}
$(x'_k,y'_k)=
\begin{cases}
(y_0,x_0),&\text{if }k=0;\\
(y_{k-1},x_{k-1}),&\text{otherwise}
\end{cases}$ 
(see \cref{ss:ACJ}). 
\item[\textbf{DR}:] 
200 iterations, with $p=1$, $\alpha=5$, and $\lambda=1$
(see \cref{ss:DR}). 
\item[\textbf{FDPG}:] 
200 iterations, with $\alpha =1$, $L=16$, and $\varepsilon=1/4$ 
(see \cref{ss:DPG}). 
\item[\textbf{accPDA}:] 
200 iterations, with 
$\alpha = 1$, $\rho_0=1$ and $\rho_{\text{max}}=100000 $ 
(see \cref{ss:PDA}). 
\item[\textbf{SSD}:] 
1800 iterations, with 
$\alpha=1$, $L=1$, and $\eta_k=1/\sqrt{k+1}$ 
(see \cref{ss:SSD}). 
\end{enumerate}

In each case, we use the starting point $x_0=y_0=(8,-13)$. The distance of the iterates from the solution, 
calculated using \cref{e:knownsol},
results in the convergence plot shown in \cref{f:toycvg},
where the grey-dotted lines marks intervals of 50 iterations of \textbf{DR}.
From the plot, we see that \textbf{accPDA} appears to perform better than
\textbf{DR} for the first 50 iterations, 
but then \textbf{DR} slowly but steadily begins to produce 
the most accurate solution.
We note that since \textbf{FDPG} is solving for a strong convex version of the objective function $\|x-y\|$,
it converges to a solution that is not the same as our original problem.
\end{example}

Note that \textbf{ACJ} {and \textbf{SSD}} do not perform as well as 
the other algorithms in \cref{eg:toy};
however, when $m$ becomes larger, 
they become much more competitive
as the following example illustrates:

\begin{figure}[ht]
\centering
\includegraphics[width=\textwidth]{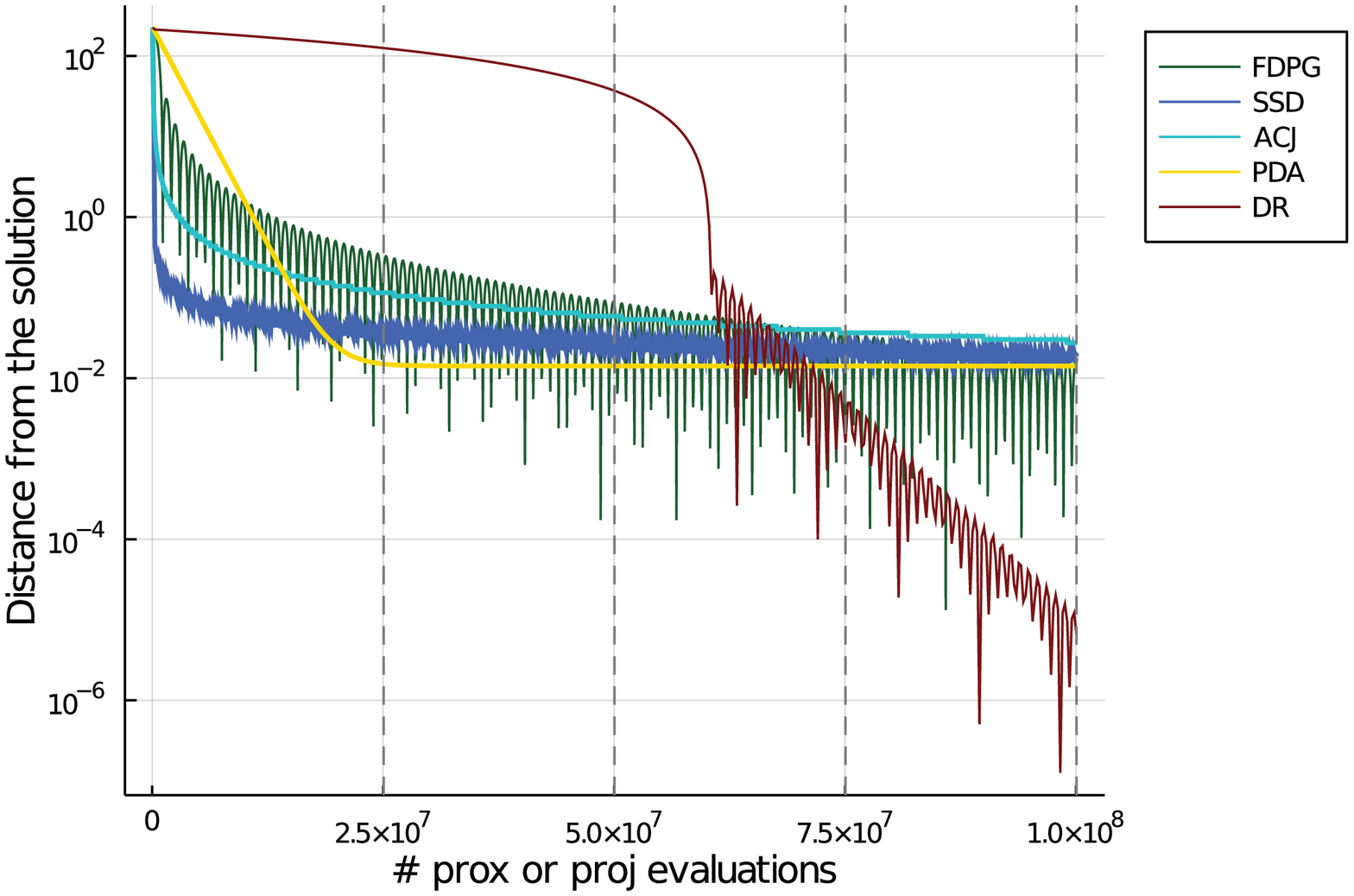} 
\caption{\cref{eg:bigtoycvg} for which $n=m=1000$}
\label{f:bigtoycvgn1000}
\end{figure}

\begin{example}

\label{eg:bigtoycvg}

Let $n\in\{1,2,\ldots\}$, and 
consider the subsets $A$ and $B$ of $Y=\RR^n$, defined by 
the two system of $m=n$ linear inequalities 
\begin{equation}
\begin{bmatrix}x_1\\\vdots\\x_n\end{bmatrix}\geq \begin{bmatrix}5\\\vdots\\5\end{bmatrix}
\quad \text{and} \quad 
\begin{bmatrix}x_1\\\vdots\\x_n\end{bmatrix}\leq \begin{bmatrix}-5\\\vdots\\-5\end{bmatrix},
\end{equation}
respectively.
Clearly, the unique solution to the corresponding problem 
\cref{e:P} is $(\bar{x},\bar{y})\in Y\times Y$, where 
$\bar{x}=(5,\ldots,5)$ and $\bar{y}=(-5,\ldots,-5)$. 
This time we run the algorithms for around $10^8$ prox evaluations 
with the starting point $x_0=y_0=(0,\dots,0)$. 
It turns out that in this case, \textbf{PDA} outperforms \textbf{accPDA}, 
so we omit the results of the latter. 
We only use the error metric once after every 50000 evaluations. 
We used the following parameters and also report to how many iterations in each algorithm this corresponded: 

\begin{enumerate}[leftmargin=0cm,itemindent=2.5cm]
\item[\textbf{ACJ}:] 
169 iterations, with
$\lambda_k={1}/{(k+1)}$, $n_k=\lfloor 1.1^k \rfloor$, and

\hspace{2.5cm}
$(x'_k,y'_k)=
\begin{cases}
(y_0,x_0),&\text{if }k=0;\\
(y_{k-1},x_{k-1}),&\text{otherwise}
\end{cases}$ 
(see \cref{ss:ACJ}). 
\item[\textbf{DR}:] 
50000 iterations, with $p=1$, $\alpha=5$, and $\lambda=1$
(see \cref{ss:DR}). 
\item[\textbf{FDPG}:] 
50000 iterations, with $\alpha =1$, $L=10000$, and $\varepsilon=1/10$ 
(see \cref{ss:DPG}). 
\item[\textbf{PDA}:] 
50000 iterations, with 
$\alpha = 1$, $\rho_0=1$ and $\rho_{\text{max}}=10^5 $ 
(see \cref{ss:PDA}). 
\item[\textbf{SSD}:] 
$\approx 10^8$ iterations, with 
$\alpha=1$, $L=10$, and $\eta_k=1/\sqrt{k+1}$ 
(see \cref{ss:SSD}). 
\end{enumerate}
For $n=1000 = m$ and the error metric given by 
\cref{e:knownsol},
we obtain the plot shown in \cref{f:bigtoycvgn1000}. 
Note that \textbf{ACJ} and \textbf{SSD} operate in 
$X = Y\times Y = \RR^{2000}$ while, 
for instance, $\textbf{DR}$ operates
in the much bigger space  $X^{m+1}=\RR^{2001000}$!
The plot makes it clear that in this 
situation \textbf{ACJ} and \textbf{SSD} fare much better
than in \cref{eg:toy}.

We note that given enough iterations, 
\textbf{DR} again trumps the other algorithms,
but the initial descent is very slow. 
In fact,the other algorithms perform much better in the beginning than \textbf{DR}. 
This suggest an interesting topic for further research:
one could consider a hybrid approach, 
where one uses an algorithm such as \textbf{SSD},
\textbf{ACJ}, \textbf{PDA} or \textbf{FDPG}, 
and then switches over to \textbf{DR}. 
{ Note that only \textbf{ACJ} and \textbf{DR} are known
to converge to a solution of the original --- the algorithms \textbf{SSD}, \textbf{PDA}, 
and \textbf{FDPG} solve perturbed versions and thus can play a role in quickly getting ``close'' to nearby points  in a preprocessing capacity.}
\end{example}

\subsection{What to do in the absence of known solutions}

\label{ss:unknown}

\begin{figure}[H]
\centering
\includegraphics[width=0.9\textwidth]{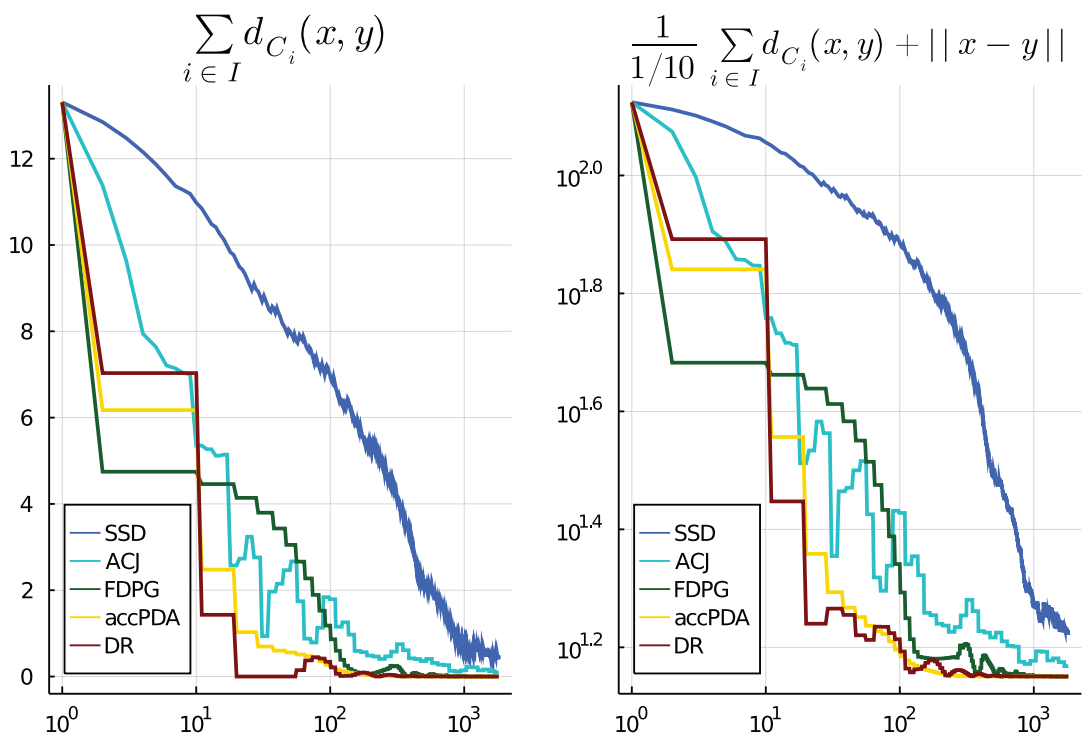}
\caption{Using \cref{e:unknownmetric} to measure performance of the algorithms}
\label{f:unknownmetric}
\end{figure}

In general, one has no access to true solutions, 
so it becomes necessary to measure performance by a metric 
different from \cref{e:knownsol}.  
We propose the measure 
\begin{equation}
\label{e:unknownmetric}
D_\delta(x,y) :=\|x-y\|+\frac{1}{\delta}\sum_{i\in I} 
d_{C_i}(x,y),
\end{equation}
where $\delta>0$. 
Because the problem asks to find a point in $A\times B$,
feasibility is of greater importance than 
minimizing $\|x-y\|$; 
thus, a smaller value of $\delta$ is 
desirable to stress feasibility. 
(One could also envision a ``dynamic'' metric, 
where $\delta\to 0$ as the number of iterations increases,
but we have not tested this.) 
Revisiting the problem considered in \cref{eg:toy},
we show in \cref{f:unknownmetric} the convergence plot 
using the parameter $\delta=1/10$. 
This time, the horizontal axis was taken with a log scaling to
increase readability of the resulting graph. 
The behaviour of the algorithms in the plot on the right, 
which includes the feasibility conditions along with the distance between $x$ and $y$, resembles the one seen in \cref{f:toycvg}.

\subsection{Combining constraints}
\label{sec:duo}

\begin{figure}[H]
\centering
\includegraphics[width=0.9\textwidth]{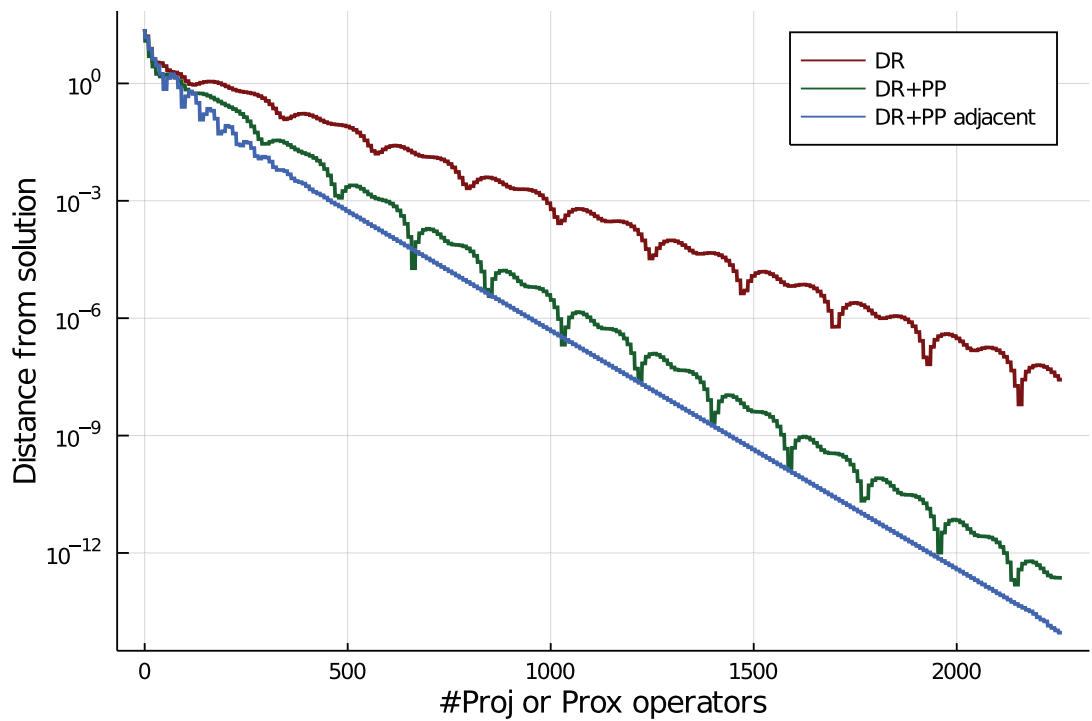}
\caption{Comparing variants of \textbf{DR} on \cref{eg:toy} with paired projections}
\label{f:DuoDR}
\end{figure}

In some cases, it is possible to combine constraints and still
be able to compute the projection onto the intersection. 
For instance, if all sets $A_i$ are halfspaces,
then any two sets $A_i$ and $A_j$ may be combined and
the projection onto the intersection is explicitly available 
using, e.g., \cite[Proposition~2.22--2.24]{BC2017}.

Revising \cref{eg:toy} in this light, 
we ran \textbf{DR} with the paired projection (PP), 
with choosing mostly non-adjacent halfspaces 
(labelled as \texttt{DR+PP} in \cref{f:DuoDR}), 
and also with choosing explicitly adjacent halfspaces 
(labelled as \texttt{DR+PP adjacent} in \cref{f:DuoDR}) 
along with the original version of \textbf{DR}
(labelled as \texttt{DR} in \cref{f:DuoDR}) 
To compare this fairly, 
we count one ``paired'' projection (onto the intersection 
of two halfspaces) as being equivalent to two regular projections. 
The convergence plot shown in \cref{f:DuoDR} illustrates that 
adding the paired projections significantly 
improves the performance significantly, 
even more so when the halfspaces are adjacent. 
The well known and characteristic ``rippling'' seen in 
typical \textbf{DR} curves is heavily damped in the last case.

\begin{figure}[h]
   \centering
\includegraphics[width=\textwidth]{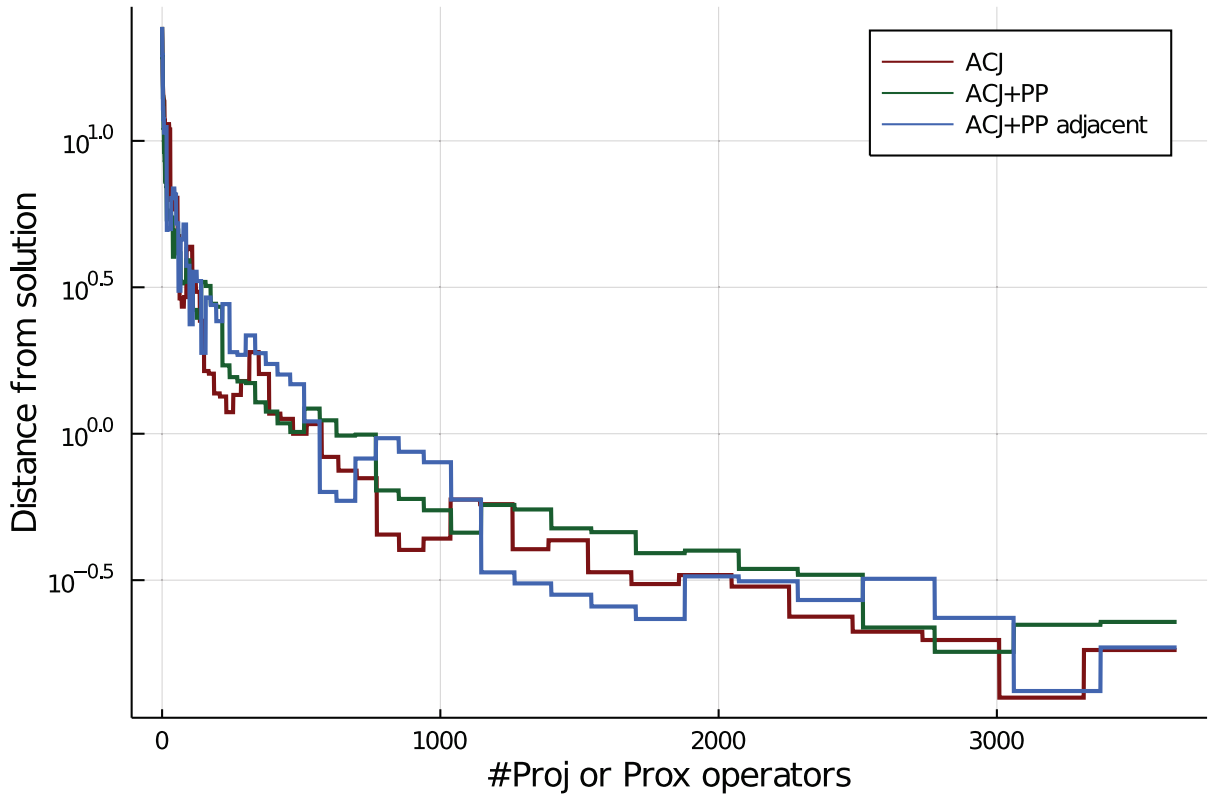}
   \caption{Comparing \textbf{ACJ} on \cref{eg:toy} with paired projections}
   \label{f:DuoACJ}
\end{figure}
\begin{figure}[htbp]
   \centering   \includegraphics[width=0.8\textwidth]{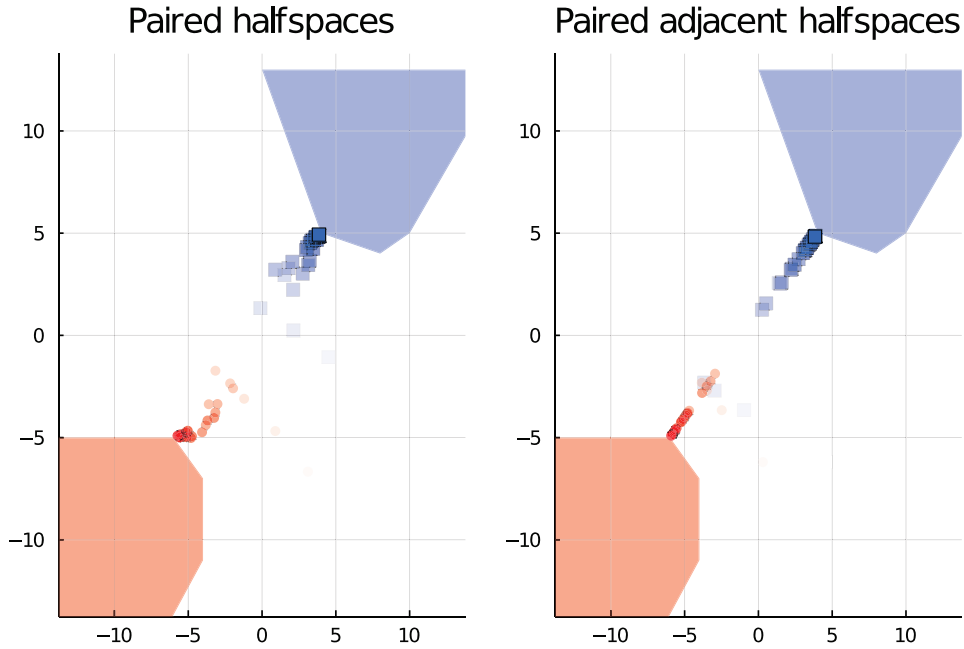}
   \caption{Appearance of the iterations of \textbf{ACJ} for the paired projections}
   \label{f:DuoACJPlot}
\end{figure}

Using this technique on \textbf{ACJ},
we note that the paired-projection variants 
do not improve the performance significantly; 
see \cref{f:DuoACJ}. 
On the other hand, the approach of the iterates to the true solution looks far less scattered as can be seen in 
\cref{f:DuoACJPlot}. 

In higher dimensions, 
further investigations are needed to determine the ``best'' way 
to pair up halfspaces as ``adjacent halfspaces''.

\section{Conclusion}
\label{sec:final}

We revisited the recent study by Aharoni, Censor, and Jiang
on finding best approximation pairs of two polyhedra.
The framework we proposed works for two sets that are themselves
finite intersections of closed convex sets with ``simple'' projections. 
Several algorithms were proposed and the required
prox operators were computed. Our numerical experiments suggested
that other algorithms deserve serious consideration.

\section*{Acknowledgments}
{The authors thank the editor and the reviewers for 
helpful suggestions and constructive feedback which helped us to
improve the presentation of the results, and Patrick Combettes 
for pointing out the relevant reference \cite{CDV}.}  
HHB and XW are supported by the Natural Sciences and
Engineering Research Council of Canada.


\end{document}